\documentclass[reqno]{amsart}
\usepackage[]{latexsym}
\usepackage[]{amssymb}
\usepackage[]{amsmath}
\usepackage{amsfonts}
\usepackage{amsthm}
\usepackage{graphicx}
\usepackage[all]{xy}
\usepackage[]{mathrsfs}
\usepackage{enumerate}
\usepackage{color}
\usepackage[capitalise]{cleveref}

\usepackage[symbol]{footmisc}

\newtheorem{theorem}{Theorem}[section]

\newtheorem{thm}[theorem]{Theorem}
\newtheorem{cor}[theorem]{Corollary}
\newtheorem{prop}[theorem]{Proposition}
\newtheorem{lemma}[theorem]{Lemma}

\theoremstyle{definition}

\newtheorem{question}[theorem]{Question}
\newtheorem{remark}[theorem]{Remark}

\numberwithin{equation}{theorem}

\newenvironment{sproof}{%
  \proof}{\endproof}

\DeclareMathAlphabet{\mathpzc}{OT1}{pzc}{m}{it}

\makeatletter
\newcommand{\bigperp}{%
  \mathop{\mathpalette\bigp@rp\relax}%
  \displaylimits
}

\newcommand{\bigp@rp}[2]{%
  \vcenter{
    \m@th\hbox{\scalebox{\ifx#1\displaystyle2.1\else1.5\fi}{$#1\perp$}}
  }%
}
\makeatother
\begin{document}

\newcommand{\str}{{\mathtt{str}}}
\newcommand{\Trd}{{\mathrm{Trd}}}
\newcommand{\rad}{{\mathrm{rad}}}
\newcommand{\id}{{\mathrm{id}}}
\newcommand{\Ad}{{\mathrm{Ad}}}
\newcommand{\Ker}{{\mathrm{Ker}}}
\newcommand{\wedges}[1]{\d b_1\wedge\ldots\wedge \d b_{#1}}
\newcommand{\rank}{{\mathrm{rank}}}
\renewcommand{\dim}{{\mathrm{dim}}}
\newcommand{\coker}{{\mathrm{Coker}}}
\newcommand{\can}{\overline{\rule{2.5mm}{0mm}\rule{0mm}{4pt}}}
\newcommand{\End}{{\mathrm{End}}}
\newcommand{\Sand}{{\mathrm{Sand}}}
\newcommand{\Hom}{{\mathrm{Hom}}}
\newcommand{\Nrd}{{\mathrm{Nrd}}}
\newcommand{\Alt}{{\mathrm{Alt}}}

\newcommand{\Srd}{{\mathrm{Srd}}}
\newcommand{\ad}{{\mathrm{ad}}}
\newcommand{\rk}{{\mathrm{rk}}}
\newcommand{\Mon}{{\mathrm{Mon}}}
\newcommand{\disc}{{\mathrm{disc}}}
\newcommand{\Sym}{{\mathrm{Sym}}}
\newcommand{\Skew}{{\mathrm{Skew}}}
\newcommand{\Nrp}{{\mathrm{Nrp}}}
\newcommand{\Trp}{{\mathrm{Trp}}}
\newcommand{\Symd}{{\mathrm{Symd}}}
\renewcommand{\dir}{{\mathrm{dir}}}
\renewcommand{\geq}{\geqslant}
\renewcommand{\leq}{\leqslant}
\newcommand{\an}{{\mathrm{an}}}
\renewcommand{\Im}{{\mathrm{Im}}}
\newcommand{\Int}{{\mathrm{Int}}}
\renewcommand{\d}{{\mathrm{d}}}
\newcommand{\coind}{{\mathrm{coind}}}

\newcommand{\qp}[2]{\mbox{$#1 \otimes^\mathrm{qp}#2 $}}
\newcommand{\qf}[1]{\mbox{$\langle #1\rangle $}}
\newcommand{\pff}[1]{\mbox{$\langle\!\langle #1
\rangle\!\rangle $}}
\newcommand{\pfr}[1]{\mbox{$\langle\!\langle #1 ]]$}}
\newcommand{\HH}{{\mathbb H}}
\newcommand{\s}{{\sigma}}
\newcommand{\lra}{{\longrightarrow}}
\newcommand{\ZZ}{{\mathbb Z}}
\newcommand{\nat}{{\mathbb N}}
\newcommand{\FF}{{\mathbb F}}
\newcommand{\PERP}{\mbox{\raisebox{-.5ex}{{\Huge $\perp$}}}}
\newcommand{\Perp}{\mbox{\raisebox{-.2ex}{{\Large $\perp$}}}}
\newcommand{\M}[1]{\mathbb{M}( #1)}

\newcommand{\vf}{\varphi}
\newcommand{\mg}[1]{#1^{\times}}

\renewcommand{\thefootnote}{\fnsymbol{footnote}}

\title{Generalised quadratic forms and the  {\MakeLowercase{\emph{u}}}--invariant}
\author{Andrew Dolphin} \date{}
\address{Departement Wiskunde--Informatica, Universiteit Antwerpen, Belgium/\newline 
Department of Mathematics, Ghent University, Belgium
}
\email{Andrew.Dolphin@uantwerpen.be} 

\begin{abstract} 
The $u$-invariant of a field is  the supremum of the dimensions of anisotropic quadratic forms over the field. We define  corresponding  $u$-invariants for hermitian and generalised quadratic forms over a division algebra with involution in characteristic $2$ and investigate the relationships between them. 
We  investigate  these invariants in particular in the case of a quaternion algebra  and futher when this quaternion algebra is the unique   quaternion division algebra over a field. 
 \end{abstract}

\maketitle

\section{Introduction}

The $u$-invariant is a classical invariant of a field. For a field $F$, the invariant $u(F)$ is defined as the supremum of the dimensions of anisotropic quadratic forms  over $F$. The $u$-invariant  is a much studied topic, and results include the celebrated  proof of Merkurjev in \cite{Merkurjev:evenuinv} that for every even number there exists a field with $u$-invariant of that value,  disproving a long standing conjecture of Kaplansky that the $u$-invariant must be a $2$-power or infinite.  See \cite{Hoffmann:usurvy} for a survey of this and related questions.

The $u$-invariants of fields of characteristic $2$ have also been studied. It was noted in   \cite{Baeza:uinv2} that is is convenient to  consider two different $u$-invariants,  one related to anisotropic quadratic forms in general, and one related to the anisotropy of nonsingular  quadratic forms. In characteristic different from $2$ these invariants are equal, as any singular quadratic forms over such fields are isotropic. In characteristic $2$, however, they can hold very different values,  see for example \cite{mammonetignolwad:char2u}. In characteristic $2$, one can also consider  similar invariants  related to the anisotropy of symmetric  bilinear forms, rather than quadratic forms.
 Again, this invariant is equal to the usual $u$-invariant for nonsingular quadratic forms in characteristic different from $2$, as here symmetric bilinear forms and quadratic forms are equivalent objects.

In \cite{Mahmoudi:heruinv}  a  hermitian version of the $u$-invariant  was introduced in characteristic different from $2$. If $(D,\theta)$ is a division algebra over $F$ with involution $\theta$, then $u(D,\theta)$ is the  supremum of the dimensions of anisotropic hermitian forms over $(D,\theta)$. 
This definition 
can also be applied in characteristic $2$. 
However, as for fields,  there is  a large variety of invariants analogous to this $u$-invariant to investigate in characteristic $2$. These  include other invariants related to hermitian  forms, but also invariants related to
 generalised quadratic forms.
 Generalised quadratic forms (also known as pseudo-quadratic forms) are an extension of the concept of quadratic forms over a field to the setting of central simple division algebras with involution, first introduced in \cite{Tits:genquadforms}. They generalise quadratic forms to this setting in an analogous manner to the generalisation of symmetric  bilinear forms to hermitian forms (see \cref{basics}). In particular, they are used to study twisted orthogonal groups in characteristic $2$ much as hermitian forms are used to study twisted orthogonal groups in characteristic different from $2$.
 So far the study of generalised quadratic forms has been largely restricted  to nonsingular quadratic forms, with  singular forms mainly only being considered over fields (see, for example, \cite{HoffmannLaghribi:qfpfisterneigbourc2}).
 Here we study $u$-invariants related both to  nonsingular generalised quadratic forms as well as to  forms that may be singular, following on from work begun in \cite{dolphin:singuni}. It is also convenient to consider $u$-invariants related to hermitian forms of specific types, specifically alternating and direct hermitian forms (the latter were introduced in \cite{dolphin:direct}).

After some general preliminary results, we introduce  various  invariants for generalised quadratic and hermitian forms in characteristic $2$ in \cref{uinvars} and investigate their properties and interrelations.
In \cref{directandtotdecomp} we calculate  our $u$-invariant related to direct hermitian forms for division algebras that are a tensor product of quaternion algebras, generalising   a  result for symmetric bilinear forms over fields (see \cref{splitcase}). 
 We  consider the values of our invariants for $F$-quaternion algebras in \cref{quatalgs} and determine them when there is a unique square class in $F$. We 
 further investigate our invariants 
 in the case where the given quaternion algebra is the unique division algebra over a field in \cref{kapfields}. 
We aim to give bounds on our invariants in terms of the degree of the field extension $F/F^2$ (cf.~\cite[(1.2)]{Baeza:uinv2}).

While our definitions for the  invariants we study are still valid with no assumption on the characteristic of the underlying field, we always assume that we are in characteristic $2$. 
This is mainly for  clarity's sake, as 
 many of the different  $u$-invariants  are equivalent  in characteristic different from $2$ and yet can have wildly different values in characteristic $2$. Moreover several of our main results are substantially  different from their general characteristic different from $2$ analogues (see for example \cref{charnot2}).  

\section{Algebras with involution}

For two objects $\alpha$ and $\beta$ in a certain category, we write $\alpha\simeq\beta$ to indicate that they are {isomorphic}, i.e.~that there exists an isomorphism between them.
This applies to algebras with involution, but also to quadratic  and hermitian   forms, where the corresponding isomorphisms are  called {isometries}.

Throughout let $F$ be a field of characteristic $2$. 
We refer to \cite{pierce:1982} as a general reference on finite-dimensional algebras over fields, and for central simple algebras in particular, and to \cite{Knus:1998} for involutions.
Let $A$ be an (associative) $F$-algebra.  
We denote the group of invertible elements in $A$ by $A^\times$. If $A^\times=A\setminus\{0\}$ we say that $A$ is \emph{division}.
We denote the centre of $A$ by $Z(A)$.
For a field extension $K/F$, the $K$-algebra $A\otimes_F K$ is denoted by $A_K$.
 An $F$-\emph{involution on $A$} is an $F$-linear map $\sigma:A\rightarrow A$ such that  $\sigma(xy)=\sigma(y)\sigma(x)$ for all $x,y\in A$ and  $\sigma^2=\textrm{id}_A$. 

Assume now that $A$ is finite-dimensional and simple  (i.e.~it has no non-trivial  two-sided ideals).
By Wedderburn's Theorem (see \cite[(1.1)]{Knus:1998}), $A\simeq\End_D(V)$ for a finite-dimensional $F$-division algebra $D$  and a finite-dimensional right $D$-vector space $V$.
 The centre of $A$ is a field and   $\dim_{Z(A)}(A)$ is a square number, whose positive square root is called the \emph{degree of $A$} and is denoted $\mathrm{deg}(A)$. The degree of $D$ is called the \emph{index of $A$} and denoted $\mathrm{ind}(A)$. We call $A$ \emph{split} if $\mathrm{ind}(A)=1$, that is $A\simeq \End_F(V)$ for some finite-dimensional right $F$-vector space $V$. 
The \emph{coindex of $A$} is $\mathrm{deg}(A)/\mathrm{ind}(A)$ and is denoted $\mathrm{coind}(A)$.  
If $Z(A)=F$, then we call the $F$-algebra $A$ \emph{central simple}.
Two central simple $F$-algebras $A$ and $A'$ are called \emph{Brauer equivalent} if $A$ and $A'$ are isomorphic to endomorphism algebras of two right vector spaces over the same $F$-division algebra.  In this case we write $A\sim_{B} A'$.


  Let $\Omega$ be an algebraic closure of $F$.  
  By Wedderburn's Theorem,  under scalar extension to $\Omega$ every central simple $F$--algebra of degree $n$ becomes isomorphic to $M_n(\Omega)$, the algebra of $n\times n$ matrices over $\Omega$. Therefore if $A$ is a central simple $F$--algebra we may fix an $F$--algebra embedding $A\rightarrow M_n(\Omega)$ and view every element $a\in A$ as a matrix in $M_n(\Omega)$. The characteristic polynomial  of this matrix has coefficients in $F$ and is independent of the embedding of $A$ into $M_n(\Omega)$ (see  \cite[\S 16.1]{pierce:1982}). We call this polynomial the \emph{reduced characteristic polynomial of $A$} and denote it by 
  $ \mathrm{Prd}_{A,a}= X^n- s_1(a)X^{n-1}+s_2(a)X^{n-2}-\ldots+(-1)^ns_n(a).$
  We call $s_1(a)$ the
   \emph{reduced trace of $a$} and  $s_n(a)$ the \emph{reduced norm of $a$} and denote them by $\Trd_A(a)$ and $\Nrd_A(a)$ respectively. We also denote $s_2(a)$ by $\Srd_A(a)$.

By an \emph{$F$-algebra with involution} we mean a pair $(A,\sigma)$ of a finite-dimensional central simple $F$-algebra $A$ and an $F$-involution $\sigma$ on $A$ (note that we only consider involutions that are  linear with respect to the centre of $A$, that is involutions of the first kind, here. See \cref{othercases}).  We use the following notation:
$$\mathrm{Sym}(A,\sigma)= \{ a\in A\mid\sigma(a)=a\}
 \textrm{ and }\,
\Alt(A,\sigma)= \{ a-\sigma( a)\,|\, a\in A\}\,.$$
These are $F$-linear subspaces of $A$.
For a field extension $K/F$ we write $(A,\s)_K=(A\otimes_F K,\s\otimes\mathrm{id})$.
We call $(A,\sigma)$ \emph{isotropic}  if  there exists an element $a\in A\setminus\{0\}$ such that  $\sigma(a)a=0$, and \emph{anisotropic} otherwise. In particular if $A$ is division, then $(A,\s)$ is anisotropic. 
We call $(A,\s)$ \emph{direct} if for $a\in A$, $\s(a)a\in\Alt(A,\s)$ implies $a=0$.  Note that direct $F$-algebras with involution are always anisotropic as $0\in\Alt(A,\s)$.

\begin{prop}[{\cite[(9.3)]{dolphin:direct}}]\label{sepdir}
Let $(A,\s)$ be an $F$-algebra with  involution and let $K/F$ be a separable algebraic  field extension such that $A_K$ is split. Then $(A,\s)$ is direct if and only if $(A,\s)_K$ is anisotropic. 
\end{prop}

Let  $(A,\sigma)$ and $(B,\tau)$ be $F$--algebras with involution.
By a \emph{homomorphism of algebras with involution}  $\Phi:(A,\sigma)\rightarrow(B,\tau)$ we mean an $F$--algebra homomorphism $\Phi: A\rightarrow B$ satisfying $\Phi\circ\sigma=\tau\circ\Phi$. If there exists an isomorphism  $\Phi:(A,\sigma)\rightarrow(B,\tau)$, then we say  the $F$--algebras with involution $(A,\sigma)$ and $(B,\tau)$ are \emph{isomorphic}. 
Letting $(\sigma\otimes\tau)(a\otimes b)=\sigma(a)\otimes \tau(b)$ for $a\in A$ and $b\in B$ determines an $F$--involution  $\sigma\otimes \tau$ on the $F$--algebra $A\otimes_F B$. We denote the pair $(A\otimes_F B,\sigma\otimes \tau)$ by $(A,\sigma)\otimes(B,\tau)$.

\section{Quadratic and hermitian forms}\label{basics}

In this section we recall the basic terminology and results we use from hermitian and quadratic form theory. 
We refer to \cite[Chapter 1]{{Knus:1991}} as a general reference on hermitian and quadratic forms.

Let $(D,\theta)$ be an $F$-division algebra.
A \emph{hermitian  form over $(D, \theta)$} is a pair $(V,h)$ where $V$ is a finite-dimensional right $D$-vector space and $h$ is
 a bi-additive map  $h:V\times V\rightarrow D$ such that $h(xd_1,yd_2)=\theta(d_1)h(x,y)d_2$ and 
$h(y,x)=\theta({h(x,y)})$   for all  $x,y\in V$ and  $d_1,d_2\in D$.
 The \emph{radical of $(V,h)$} is the set $$\rad(V,h)=\{x\in V\mid h(x,y)=0 \textrm{ for all } y\in  V\}\,.$$
 We say that $(V,h)$ is
\emph{degenerate} if $\rad(V,h)\neq \{0\}$ and  \emph{nondegenerate} otherwise.

Let $\varphi=(V,h)$ be a hermitian form over $(D,\theta)$.
We call $\dim_D(V)$ the \emph{dimension of  $\varphi$} and denote it by $\dim_D(\varphi)$.
We say $\varphi$ \emph{represents an element $a\in D$} if $h(x,x)=a$ for some $x\in V\setminus\{0\}$. 
We call $\varphi$ \emph{isotropic} if it represents $0$ and \emph{anisotropic} otherwise.
 We say that $\varphi$ is \emph{alternating} if $h(x,x)\in \Alt(D,\theta)$ for all $x\in V$ and \emph{direct} if $h(x,x)\notin \Alt(D,\theta)$  for all $x\in V\setminus \{0\}$. Note that direct hermitian forms are always anisotropic. 
 For $u\in D^\times$ we write $uh$ for the map $(x,y)\mapsto uh(x,y)$ for all $x,y\in V$ and if $u\in F^\times$ we write $u\varphi$ for the hermitian  form $(V,uh)$ over $(D,\theta)$.

Let $\varphi_1=(V,h_1)$ and $\varphi_2=(W,h_2)$ be hermitian forms over $(D,\theta)$.
By an \emph{isometry of hermitian forms $\phi:\varphi_1\rightarrow\varphi_2$} we mean an isomorphism of $D$-vector spaces $\phi:V\longrightarrow W$ such that $h_1(x,y)=h_2(\phi(x),\phi(y))$ for all $x,y\in V$.
We say that $\varphi_1$ and $\varphi_2$ are \emph{similar} if $\varphi_1\simeq c\varphi_2$ for some $c\in F^\times$.
The \emph{orthogonal sum} of  $\varphi_1$ and $\varphi_2$ is defined to be the pair $(V\times W, h)$ where the $F$-linear map  $h:(V\times W)\times (V\times W)\rightarrow D$ is given by $h((v_1, w_1),(v_2, w_2))= h_1(v_1,v_2) +h_2(w_1,w_2)$ for any $v_1, v_2\in V$ and $w_1,w_2\in W$; we  denote it by  $ \varphi_1\perp \varphi_2$.

\begin{prop}[{\cite[(5.6)]{dolphin:direct}}]\label{direct}
Let $\varphi$ and $\psi$ be anisotropic hermitian forms over $(D,\theta)$ such that $\varphi$ is direct and $\psi$ is alternating. Then $\varphi\perp \psi$ is anisotropic.
\end{prop}

 For $a_1,\ldots,a_n\in D\cap\,\mathrm{Sym}(D,\theta)$, we denote by $\qf{a_1,\ldots,a_n}_{(D,\theta)}^{h}$ the hermitian form $(D^n,h)$ over $(D,\theta)$ where  $h:D^n\times D^n\rightarrow D$ is given by $$(x,y)\mapsto \sum_{i=1}^n \theta( x_i)a_iy_i\, .$$

A \emph{symmetric bilinear form over $F$} is a hermitian form over $(F,\id)$. 
Let $\varphi=(V,b)$ be a symmetric bilinear form over $F$ and  $\psi=(W,h)$ be a hermitian form over   $(D,\theta)$. Then $V\otimes_FW$ is a finite dimensional right $D$-vector space. Further, 
 there is a unique $F$-bilinear map $b\otimes h:(V\otimes_F W)\times (V\otimes_F W)\rightarrow F$ such that 
$(b\otimes h)\left( (v_1\otimes w_1), (v_2\otimes w_2)\right) =b(v_1,v_2)\cdot h(w_1,w_2) $
for all $w_1,w_2\in W, v_1,v_2\in V$. We call the hermitian form $(V\otimes_F W, b\otimes h)$ over $(D,\theta)$ the  \emph{tensor product of $\varphi$ and $\psi$}, and denote it by $\varphi\otimes \psi$ (cf.~\cite[Chapter 1, \S8]{Knus:1991}).
Let $K/F$ be a field extension. Then we write $(V,b)_K=(V\otimes_F K, b_K) $ where $b_K$ is  given by $b_K( v\otimes k , w\otimes k')= kk'b(v,w)$ for all $v,w\in V$ and $k,k'\in K$.  
For a positive integer $m$, by an \emph{$m$-fold bilinear Pfister form over $F$} we mean a nondegenerate symmetric bilinear form over $F$ that is isometric to
 $$\qf{1,a_1}^h_{(F,\id)}\otimes\ldots\otimes\qf{1,a_m}^h_{(F,\id)}$$  for some 
 $a_1,\ldots,a_m\in F^\times$ and we denote this form by $\pff{a_1,\ldots, a_m}$.  We call $\qf{1}^h_{(F,\id)}$  the \emph{$0$-fold bilinear Pfister form}.

Let   $(V,h)$ be a nondegenerate hermitian form over $(D,\theta)$.
By \cite[(4.1)]{Knus:1998}, there is a unique $F$--involution $\sigma$ on $\End_D(V)$  such that 
$$h(f(x),y)=h(x,\sigma(f)(y)) \textrm{ for all } x,y\in V \textrm{ and } f\in \End_D(V).$$
We call $\sigma$ the \emph{involution adjoint to} $h$ and denote it by $\ad_h$, and we  write $\mathrm{Ad}(V,h)=(\End_D(V),\ad_h)$. By \cite[(4.2)]{Knus:1998}, for any $F$--algebra with involution $(A,\sigma)$ such that $A\sim_BD$,
there exists an  $F$-division algebra  with involution $(D,\theta')$   and
 a  hermitian form $\varphi$ over $(D,\theta')$  such that $\Ad{(\varphi)}\simeq (A,\sigma)$. The hermitian form $\varphi$ over $(D,\theta')$ is uniquely determined up to similarity. 
Let $L$ be a splitting field of the $F$--algebra $A$. An involution $(A,\sigma)$ is said to be \emph{symplectic} if $(A,\sigma)_L\simeq \mathrm{Ad}(\psi)$ for some  alternating bilinear form $\psi$ over $L$, and \emph{orthogonal} otherwise. This definition is independent of the choice of the splitting field $L$ (see  \cite[Section 2.A]{Knus:1998}). 
\begin{prop}\label{prop:sympalt} Let $\varphi$ be a nondegenerate  hermitian form over $(D,\theta)$. 
\begin{enumerate}[$(a)$]
\item  $\Ad(\varphi)$ is isotropic if and only if $\varphi$ is isotropic.
\item $\Ad(\varphi)$ is symplectic if and only if $\varphi$ is alternating. Otherwise $\Ad(\varphi)$ is orthogonal. In particular direct  involutions are orthogonal.
\item $\Ad(\varphi)$ is direct if and only if $\varphi$ is direct. 
\item For any nondegenerate  symmetric bilinear form $\psi$ over $F$ we have $\Ad(\psi\otimes \varphi)\simeq\Ad(\psi)\otimes\Ad(\varphi)$.
\end{enumerate}
\end{prop}
\begin{proof}   For  $(a)$ see \cite[(3.2)]{dolphin:quadpairs} for the case where $(D,\theta)=(F,\id)$. The more general case is similar.
For $(b)$, see \cite[(4.2)]{Knus:1998}. For $(c)$ see \cite[(7.3)]{dolphin:direct}. For $(d)$ see \cite[(3.4)]{dolphin:totdecompsymp}.
\end{proof}

A \emph{quadratic form over $(D,\theta)$} is a pair $(V,q)$ where $V$ is a finite-dimensional right $D$-vector space and $q$ is a map $q:V\rightarrow D/\Alt(D,\theta)$ subject to the following conditions:
\begin{enumerate}[$(a)$]
\item  $q(xd)=\theta(d)q(x)d$ for all $x\in V$ and $d\in D$,
\item $q(x+y)- q(x)-q(y) = h(x,y) +\Alt(D,\theta)$ for all $x,y\in V$ and a hermitian form $(V,h)$ over $(D,\theta)$. 
\end{enumerate}
In this case the hermitian form $(V,h)$ is uniquely determined (see \cite[(1.1)]{tignol:qfskewfield}) and we call it the \emph{polar form of $(V,q)$}. Note that it follows from $(b)$ that $h(x,x)\in\Alt(D,\theta)$ for all $x\in V$, hence the polar form of any quadratic form over $(D,\theta)$ is  alternating. 
We call $(V,q)$ \emph{nonsingular} if  its polar form is nondegenerate and \emph{singular} otherwise.
If the polar form of $(V,q)$ is identically zero we call $(V,q)$ \emph{totally singular}.
We call $\dim_D(V)$ the \emph{dimension of  $(V,q)$} and denote it by $\dim_D(V,q)$.
We say that $(V,q)$   \emph{represents} an element $a\in D$ if $q(x)=a+\Alt(D,\theta)$ for some $x\in V\setminus \{0\}$.
  We call $(V,q)$ \emph{isotropic} it represents $0$  and \emph{anisotropic} otherwise.

\begin{remark}\label{othercases}
Quadratic and hermitian forms can be similarly defined if the characteristic of the base field is different from $2$ or if the involution is not linear with respect to the centre of the division algebra (i.e.~the involution is of the second kind). In this case, a quadratic form is uniquely determined by its polar hermitian form, and these two objects are essentially equivalent (see  \cite[Chapter 1, (6.6.1)]{{Knus:1991}}).
\end{remark}

Let $\rho_1=(V,q_1)$ and $\rho_2=(W,q_2)$ be quadratic forms over $(D,\theta)$. 
By an \emph{isometry of quadratic  forms $\phi:\rho_1\rightarrow\rho_2$} we mean an isomorphism of $D$-vector spaces $\phi:V\longrightarrow W$ such that $q_1(x)=q_2(\phi(x)) +\Alt(D,\theta)$ for all $x\in V$. 
The \emph{orthogonal sum of $\rho_1$ and $\rho_2$} is defined to be pair $(V\times W,q)$ where the map $q:(V\times W)\rightarrow D/\Alt(D,\theta)$ is given by $q((v, w))= q_1(v) +q_2(w)$ for all $v\in V$ and $w\in W$, and we write  $ (V\times W, q)=\rho_1\perp \rho_2$.

\begin{remark}
In the following we refer to \cite[Chapter 1]{{Knus:1991}} and \cite{dolphin:singuni} for results on quadratic forms over a division algebra with involution. Both of these references deal with the slightly more general notion of a unitary  space. A quadratic form  over $(D,\theta)$ in our sense is a $(\lambda,\Lambda)$--unitary space over $(D,\theta)$ with $\lambda=1$ and form parameter $\Lambda=\Alt(D,\theta)$ (see \cite[Chapter 1, \S5]{Knus:1991} or  \cite[\S4]{dolphin:singuni}).
\end{remark}

%
%
%

Let $\rho$ be a quadratic form over $(D,\theta)$. By \cite[(9.2)]{dolphin:singuni} there exists quadratic forms $\rho_1$ and $\rho_2$ over $(D,\theta)$ with $\rho_1$ nonsingular and $\rho_2$ totally singular such that $\rho\simeq\rho_1\perp \rho_2$. Further, the natural numbers  $n=\dim_D(\rho_1)$ and $m=\dim_D(\rho_2)$ are determined by $\rho$. In this situation we say that $\rho$ is \emph{of type $(n,m)$}.\footnote[2]{Our notion of type is slightly different than that used for fields in, for example, \cite{HoffmannLaghribi:qfpfisterneigbourc2}. In the case of a quadratic form over $(F,\id)$, the nonsingular part of the quadratic form is always even dimensional, and in the above situation \cite{HoffmannLaghribi:qfpfisterneigbourc2} would write that $\rho$ is of type $(\frac{1}{2}n,m)$. Since we work over  division algebras that are not necessarily  fields, we may have odd dimensional nonsingular quadratic forms.}

 For $a_1,\ldots,a_n\in D\setminus \Alt(D,\theta)$, we denote by $\qf{a_1,\ldots,a_n}_{(D,\theta)}$ the quadratic form $(D^n,q)$ over $(D,\theta)$ where  $q:D^n\rightarrow D$ is given by $$x\mapsto \sum_{i=1}^n a_ix_i +\Alt(D,\theta)\, .$$ We drop the algebra with involution $(D,\theta)$ from the notation when it is clear from context.
We call such a form a \emph{diagonal form}. We call a quadratic form \emph{diagonalisable} if it is isometric to a diagonal form.

\begin{lemma}\label{diagstuff}
Assume $D$ is not a field. Then every quadratic form  over $(D,\theta)$ is diagonalisable. Further, for $a,a_1,\ldots,a_n\in D\setminus \Alt(D,\theta)$ we have 
\begin{enumerate}[$(a)$]
\item $\qf{a+d}_{(D,\theta)}\simeq \qf{a}_{(D,\theta)}$ for all $d\in \Alt(D,\theta)$.
\item $\qf{a_1,\ldots, a_n}_{(D,\theta)}$ is nonsingular if and only if $a_1,\ldots, a_n\in D\setminus\Sym(D,\theta)$. 
\item  $\qf{a_1,\ldots, a_n}_{(D,\theta)}$ is totally singular if and only if $a_1,\ldots, a_n\in \Sym(D,\theta)$. 
\end{enumerate}
\end{lemma}
\begin{proof} Every quadratic form over $(D,\theta)$ is diagonalisable by \cite[(6.3)]{dolphin:singuni}. 
For all $x\in D$ and $d\in \Alt(D,\theta)$ 
we have $\theta(x)dx\in \Alt(D,\theta)$. Hence 
 $\theta(x)(a+d)x= \theta(x)ax +\Alt(D,\theta)$. Statement $(a)$ follows.
 Statements $(b)$ and $(c)$ follow from \cite[(7.5)]{dolphin:singuni} as the orthogonal sum of nonsingular (resp.~totally singular) forms is nonsingular (resp.~totally singular).  
\end{proof}

Let $\varphi=(V,h)$ be a hermitian form over $(D,\theta)$. Then the pair $(V,q_h)$ where $q_h:D\rightarrow D/\Alt(D,\theta)$ is the map given by $q_h(x)=h(x,x)+\Alt(D,\theta)$ for all $x\in V$ is  quadratic form over $(D,\theta)$. 
We call $(V,q_h)$ the \emph{quadratic form associated to $\varphi$}.
For all $x,y\in V$ we have 
$$ q_h(x+y) -q_h(x)-q_h(y)=  h(x,y)+\theta(h(x,y))\in \Alt(D,\theta)\,. $$
Hence the polar form of $(V,q_h)$ is identically zero. That is, $(V,q_h)$ is totally singular.
\begin{prop}\label{directtotsing}
Let $\varphi$ be a hermitian form over $(D,\theta)$ and let $\rho$ be its associated quadratic form. Then $\varphi$ is direct if and only if $\rho$ is anisotropic. Moreover, any anisotropic  totally singular quadratic form over $(D,\theta)$ is the associated quadratic form to some direct hermitian form over $(D,\theta)$ 
\end{prop}
\begin{proof}
Let $\varphi=(V,h)$ and $\rho=(V,q_h)$. 
Then  $\varphi$  is not direct if and only if there exists an $x\in V\setminus\{0\}$ such that $h(x,x)=q_h(x)\in \Alt(D,\theta)$. This shows the first statement.
Let $\rho'$ be an anisotropic  totally singular quadratic form over $(D,\theta)$. Then by \cref{diagstuff} there exist $a_1,\ldots, a_n\in \Sym(D,\theta)\setminus\{0\}$ such that $\rho'\simeq \qf{a_1,\ldots, a_n}_{(D,\theta)}$.  Let $\varphi \simeq\qf{a_1,\ldots, a_n}^h_{(D,\theta)}$. Then the  quadratic form associated to $\varphi$ is $\rho'$ and in particular  $\varphi$ is direct by the first part of the proof.
\end{proof}

For $u\in D^\times$, let $\textrm{Int}(u)$ denote the {inner automorphism determined by $u$}, that is the $F$--linear map $D\rightarrow D$ given by $a\mapsto uau^{-1}$ for all $a\in D$.

\begin{prop}\label{scalinginD}
For $i=1,2$ let $(D,\theta_i)$ be an $F$--division algebra with involution.  
\begin{enumerate}[$(i)$]
\item  There exists an element $u\in \Sym(D,\theta_1)\setminus \{0\}$ such that $\theta_2= \Int(u)\circ\theta_1$.
\item For every hermitian form $(V,h)$ over $(D,\theta_2)$, the pair $(V,uh)$ is a hermitian form over $(D,\theta_2)$.  Further, $(V,h)$ is isotropic (resp.~alternating or direct)  if and only if $(V,uh)$ is isotropic (resp.~alternating or direct).
\item For every quadratic form $(V,q)$ over $(D,\theta_1)$ the pair $(V,uq)$ is a quadratic form over $(D,\theta_2)$. Further, $(V,q)$ is isotropic (resp.~singular) if and only if $(V,uq)$ is isotropic  (resp.~singular).
\end{enumerate}
\end{prop}
\begin{proof}
$(i)$ See \cite[(2.7), $(2)$]{Knus:1998}. Also by \cite[(2.7), $(2)$]{Knus:1998} we  have   
\begin{eqnarray}\label{alt}\Alt(D,\theta_2)=u\cdot \Alt(D,\theta_1)\,.\end{eqnarray}

$(ii)$ For the proof that $(V,uh)$ is a hermitian form over $(D,\theta_2)$ see \cite[(3.9)]{dolphin:direct}. That $(V,h)$ is isotropic if and only if $(V,uh)$ is isotropic is clear. That $(V,h)$ is alternating (resp.~direct)  if and only if $(V,uh)$ is alternating (resp.~direct)  follows  immediately from (\ref{alt}).

$(iii)$   First note that  by (\ref{alt})  $uq$ is a map from $V$ to $D/\Alt(D,\theta_2)$. 
Let $(V,h)$ be the polar form of $(V,q)$. 
For all $x,y\in V$ and $d\in D$ we have $$u(q(xd))=u\theta_1(d)(q(x))d=uu^{-1}\theta_2(d)(uq(x))d= \theta_2(d)(uq(x))d\,,$$
and 
$$uq(x+y)- uq(x)-uq(y) = uh(x,y) +u\cdot\Alt(D,\theta_1)\,.$$
Therefore, as $(V,uh)$ is a hermitian form over $(D,\theta_2)$ by $(ii)$,  $(V,uq)$ is a quadratic form over $(D,\theta_2)$.
It is clear that $(V,h)$ is degenerate if and only if $(V,uh)$ is degenerate and  hence  $(V,q)$ is singular if and only if $(V,uq)$ is singular.
Finally, it follows again by (\ref{alt}) that for $x\in V\setminus\{0\}$ we have $q(x)\in \Alt(D,\theta_1)$ if and only if $uq(x)\in \Alt(D,\theta_2)$.   Hence $(V,q)$ is isotropic if and only if $(V,uq)$ is isotropic. 
\end{proof}

By a  \emph{quadratic form over $F$}  we mean  a quadratic form $(V,q)$ over $(F,\id)$. 
A subspace  $U\subseteq V$  is \emph{totally isotropic} (with respect to $q$) if $q|_U=0$. If $(V,q)$ is nonsingular, then we call it \emph{hyperbolic} if there exists a totally isotropic subspace $U$ of $V$ such that $\dim_F(U)=\frac{1}{2}\dim_F(V,q)$. A quadratic form  $(W,q')$ is a \emph{subform of $(V,q)$} if there exists an injective map $t:W\lra V$ such that $q'(t(x))= q(x)$ for all $x\in W$.

For $b,c\in F$, we denote the nonsingular quadratic form $(F^2,q)$ where $q:F^2\rightarrow F$ is given by $(x,y)\mapsto bx^2+xy+cy^2$ by $[b,c]$. 
Recall the concept of a tensor product of a symmetric or alternating bilinear form and a quadratic form (see \cite[p.51]{Elman:2008}).
For a positive  integer $m$, by an \emph{$m$-fold (quadratic) Pfister form over $F$}  we mean a quadratic form that is isometric to the tensor product of a $2$-dimensional nonsingular quadratic form representing $1$ and an $(m-1)$-fold bilinear Pfister form over $F$. For $a_1,\ldots, a_{m-1}\in F^\times$ we denote the $m$-fold Pfister form $\pff{a_1,\ldots, a_{m-1}}\otimes [1,b]$ by $\pfr{a_1,\ldots, a_{m-1},b}$.
 Pfister forms are either anisotropic or hyperbolic (see \cite[(9.10)]{Elman:2008}).

%
%
%
Let $\rho=(V,q)$ be a quadratic form.
For $a\in F^\times$ we write $aq$ for the map $x\mapsto aq(x)$ for all $x\in V$ and we write $a(V,q)$ for the quadratic form $(V,aq)$ over $F$.
Let $K/F$ be a field extension. Then we write $\rho_K=(V\otimes_F K, q_K) $ where $q_K$ is  the unique quadratic map such that $q_K( v\otimes b)= b^2q(v)$ for all $v\in V$ and $b\in K$.  

\section{$u$-invariants of quadratic and hermitian forms}\label{uinvars}

We now introduce the various invariants that we study and determine certain relations between them.  As we see later, it is convenient to consider not only $u$-invariants related to hermitian and quadratic forms in general, but also to direct or alternating hermitian forms and nonsingular quadratic forms.

Let $(D,\theta)$ be an  $F$-division algebra with involution.
\cref{scalinginD} immediately gives the following.

\begin{prop}\label{noinv}
The
supremum of the dimensions of anisotropic (direct or alternating) hermitian forms over $(D,\theta)$ does not depend on the $F$-involution $\theta$. Similarly, the   
supremum of the dimensions of anisotropic (nonsingular) quadratic  forms over $(D,\theta)$ does not depend on the $F$-involution $\theta$.
\end{prop}

Given \cref{noinv}, 
we define the following hermitian form $u$-invariants of the $F$-division algebra $D$.
\begin{eqnarray*}u^+(D)&=& \mathrm{sup}\{\dim_D(\varphi)\mid \varphi  \textrm{ is an anisotropic hermitian form  over $(D,\theta)$} \}\,,\\
u^+_d(D)&=& \mathrm{sup}\{\dim_D(\varphi)\mid   \varphi \textrm{ is a direct hermitian form over $(D,\theta)$} \}\,,\\
u^-(D)&=& \mathrm{sup}\{\dim_D(\varphi)\mid \varphi  \textrm{ is an anisotropic alternating hermitian } \\  &&\qquad\qquad\qquad\qquad\qquad\qquad\qquad\qquad\qquad  \textrm{form over $(D,\theta)$} \}\,.\end{eqnarray*}
Similarly, we have the following quadratic form $u$-invariants.
\begin{eqnarray*}
u(D)&=& \mathrm{sup}\{\dim_D(\rho)\mid  \rho  \textrm{ is an anisotropic quadratic form 
 over $(D,\theta)$} \}\,,\\
\widetilde{u}(D)&=& \mathrm{sup}\{\dim_D(\rho)\mid   \rho \textrm{ is an anisotropic nonsingular quadratic } \\  &&\qquad\qquad\qquad\qquad\qquad\qquad\qquad\qquad\qquad  \textrm{form over $(D,\theta)$} \}\,.\end{eqnarray*}

\begin{remark}\label{invint} Using \cref{prop:sympalt}, the invariants $u^+(D)$ (resp.~$u^-(D)$ or $u^+_d(D)$) can be naturally reinterpreted as the supremum of all coindices of central simple $F$-algebras $A$  with $A\sim_B D$ such that there exists an anisotropic $F$-algebra with orthogonal (resp.~symplectic or direct) involution $(A,\s)$. Similarly, using the analogous relationship between nonsingular quadratic forms over an algebra with involution and quadratic pairs (see \cite[\S1]{tignol:qfskewfield} and \cite[\S5]{Knus:1998}) we can interpret $\widetilde{u}(D)$ as the supremum of all coindices of central simple $F$-algebras $A$  with $A\sim_B D$ such that there exists an anisotropic $F$-algebra with quadratic pair $(A,\s,f)$.

Note that we cannot reinterpret $u(D)$ in similar terms as 
in order to construct an adjoint quadratic pair we must assume our quadratic form is nonsingular (see \cite[(1.4)]{tignol:qfskewfield}). 
There is currently no known notion of a quadratic pair-like object that can be associated with a singular quadratic form.
\end{remark}

\begin{remark}\label{totsing}
By \cref{directtotsing} we may also reinterpret $u_d^+(D)$ as the supremum of the dimensions of anisotropic totally singular quadratic forms over $(D,\theta)$.
\end{remark}

\begin{prop}\label{always2}
Assume  $D$  is not a field. Then  $u^{+}_d(D)$ and $u^{-}(D)$ are non-zero. 
\end{prop}
\begin{proof} 
By \cite[(2.6)]{Knus:1998} there exists an element $y\in \Alt(D,\theta)\setminus\{0\}$ and an element $x\in \Sym(D,\theta)\setminus \Alt(D,\theta)$. 
The nondegenerate hermitian form $\qf{y}^h_{(D,\theta)}$ over $(D,\theta)$ is anisotropic and alternating, and the  nondegenerate hermitian form $\qf{x}^h_{(D,\theta)}$ over $(D,\theta)$ is direct (see \cite[(5.4)]{dolphin:direct}). Hence 
 $u^{+}_d(D)$ and $u^{-}(D)$ are  at least $1$.
\end{proof}

\begin{remark}\label{splitcase}
For the field $F$ we have  $u^-(F)=0$  as $\Alt(F,\id)=\{0\}$ and hence all alternating bilinear forms over $F$ are isotropic. 
For any $a\in F^\times$, the symmetric bilinear form $\qf{a}^h_{(F,\id)}$ is anisotropic, hence we always have that $u^+_d({F})\neq 0$.  Further we also have $u_d^+(F)=u^+(F)$, as a symmetric bilinear form over a field is direct if and only if it is anisotropic by \cite[(5.1)]{dolphin:direct}.  
In fact we have $u^+(F)=[F:F^2]$ by  \cref{totsing} and \cite[(1.2, $(c)$)]{Baeza:uinv2}.
\end{remark}

We now show that $u(D)$ is the largest of our invariants. If $D$ is a field, this is clear from Remarks \ref{totsing} and \ref{splitcase}. However, when $D$ is not a field the relationship between our hermitian $u$-invariants and $u(D)$ is more subtle.

\begin{prop}\label{symreps} 
Let $\rho$ be a quadratic form over $(D,\theta)$ with polar form $\psi$. Then $\psi$ is anisotropic if and only if $\rho$ does not represent any elements in $\Sym(D,\theta)$.
 In particular if $\psi$ is anisotropic, $\rho$ is  nonsingular and anisotropic.
\end{prop}
\begin{proof} First note that if $(D,\theta)=(F,\id)$ then $\psi$ is always isotropic by \cref{splitcase} and any element represented by $\rho$ is in $\Sym(D,\theta)=F$. 
Assume that $D$ is not a field. 
By \cref{diagstuff} there exists $b_1,\ldots, b_n\in D^\times$ such that $\rho\simeq \qf{b_1,\ldots, b_n}_{(D,\theta)}$. Let $a_i=b_i+\theta(b_i)$ for $i=1,\ldots, n$. Then by  \cite[(6.2)]{dolphin:singuni}, $\psi\simeq \qf{a_1,\ldots, a_n}^h_{(D,\theta)}$. 


Suppose $\psi$ is anisotropic and that  there exist $x_1,\ldots, x_n\in D$ such that 
 $$\sum_{i=1}^n\theta(x_i)b_ix_i\in \Sym(D,\theta)\,.$$
 Then we have that 
 \begin{eqnarray*}\sum_{i=1}^n\theta(x_i)a_ix_i  &=& \sum_{i=1}^n\theta(x_i)(b_i +\theta(b_i))x_i = \sum_{i=1}^n\theta(x_i)b_ix_i\ + \sum_{i=1}^n\theta(x_i)\theta(b_i)x_i\\&=& \sum_{i=1}^n\theta(x_i)b_ix_i\ + \theta\left(\sum_{i=1}^n\theta(x_i)b_ix_i\right)=0\,. \end{eqnarray*}
Hence by the anisotropy of  $\varphi$  we have that $x_i=0$ for $i=1,\ldots, n$. That is, $\rho$ does not represent any elements in $\Sym(D,\theta)$. 

Conversely, suppose $\psi$ is isotropic, that is there exist $x_1,\ldots, x_n\in D$ not all zero such that 
$$0= \sum_{i=1}^n \theta(x_i) a_i x =\sum_{i=1}^n\theta(x_i)b_ix_i\ + \theta\left(\sum_{i=1}^n\theta(x_i)b_ix_i\right)\,.$$
That is, $\sum_{i=1}^n\theta(x_i)b_ix_i\in\Sym(D,\theta)$. 
\end{proof}


\begin{cor}\label{repsym}
Let $\rho$ be a  quadratic form over $(D,\theta)$ with $\dim_D(\rho)> u^-(D)$. Then $\rho$ represents an element in $\Sym(D,\theta)$.
\end{cor}
\begin{proof}
As the polar form of $\rho$ is alternating, $\dim_D(\rho)>u^-(D)$ implies that  the polar form of $\rho$ must be isotropic. Hence $\rho$ represents an element in $\Sym(D,\theta)$ by \cref{symreps}.
\end{proof}

\begin{thm}\label{prop:start}
 Let $D$ be an $F$-division algebra.
\begin{enumerate} 
\item[$(1)$] $u^+(D)= u^+_d(D)+u^-(D)$.
\item[$(2)$] $u^-(D)\leqslant \widetilde{u}(D)$.
\item[$(3)$]  $u^+(D)\leqslant u(D)$.
\end{enumerate}
\end{thm}
\begin{proof}  Let $(D,\theta)$ be an $F$-division algebra with involution.

$(1)$ This follows immediately from \cref{direct}.

%

 $(2)$
 If $D$ is a field then $u^-(D)=0$ by \cref{splitcase}, and the result is clear.
 Otherwise let $\psi$ be an anisotropic  alternating hermitian form over $(D,\theta)$. By \cite[Chapter 1, (6.2.4)]{Knus:1991} and since $\psi$ is alternating there exists 
 $a_1,\ldots, a_n\in \Alt(D,\theta)\setminus\{0\}$ such that $\psi\simeq\qf{a_1,\ldots, a_n}^h_{(D,\theta)}$.   Let $b_1,\ldots, b_n\in D^\times$ be such that $b_i+\theta(b_i)=a_i$ for $i=1,\ldots, n$. Then by \cite[(6.2)]{dolphin:singuni}, $\psi$ is the polar form of the nonsingular quadratic form $\rho\simeq \qf{b_1,\ldots, b_n}_{(D,\theta)}$ over $(D,\theta)$.  
Further, $\rho$ is anisotropic by  \cref{symreps}.


$(3)$
Assume first that $D$ is a field. Then $u^-(D)=0$ by \cref{splitcase} and hence  $u_d^+(D)=u^+(D)$ by $(1)$. The result follows immediately in this case from   \cref{directtotsing}. Suppose now that $D$ is not a field and  there exists an $n$-dimensional direct  hermitian form  over $(D,\theta)$. Then   \cref{directtotsing} gives  an $n$-dimensional  anisotropic totally singular quadratic form $\rho$ over $(D,\theta)$.
Note that by \cref{diagstuff}, $(c)$, $\rho$ only represents elements in $\Sym(D,\theta)$. 

Suppose further that there exists an  $m$-dimensional anisotropic alternating hermitian form over $(D,\theta)$. 
As in $(2)$, using \cref{symreps}, there exists an $m$-dimensional anisotropic nonsingular  quadratic form $\rho'$ over $(D,\theta)$ which does not represent any elements in $\Sym(D,\theta)$. In particular,  the form $\rho\perp \rho'$ is anisotropic and  of dimension $n+m$. The result then follows from $(1)$.
\end{proof}

\begin{cor}\label{orth2} Assume $D$ is not a field. Then 
$u^+(D)$ and $u(D)\geqslant 2$. 
\end{cor}
\begin{proof}
This follows  from \cref{always2} and \cref{prop:start}, $(1)$ and $(3)$.
\end{proof}

\begin{remark}\label{charnot2} Let $K$ be a field of characteristic different from $2$. 
One can define $u$-invariants similar to the ones we study for hermitian and quadratic forms over a $K$-division algebra with involution $(E,\tau)$. 
Hermitian forms over $(E,\tau)$ can be either symmetric or skew-symmetric and generalised quadratic forms are equivalent to one of these types of hermitian form  (see \cite[Chapter 1, (6.5.2)]{Knus:1991}). In particular it follows that singular forms are isotropic. This results in two different $u$-invariants depending on whether one considers symmetric or skew-symmetric forms and on whether $\tau$ is orthogonal or symmetric (see \cite[(2.3)]{Mahmoudi:heruinv}). These invariants can be naturally associated with the coindices of either orthogonal or symplecitc anisotropic involutions, as in \cref{invint}, and we denote these invariants by $u^+(E)$ and $u^-(E)$ respectively (again see \cite[(2.3)]{Mahmoudi:heruinv}).


In characteristic different from $2$  the analogous  statements to \cref{prop:start} $(1)$ and $(2)$ do not hold in general. 
Let $Q$ be the unique $\mathbb{R}$-quaternion division algebra. Then $u^+(Q)=1$ by  \cite[(6.8)]{bechermah:uinvar}. However, as $u(\mathbb{R})=\infty$ it follows from  \cite[(10.1.7)]{Scharlau:1985} (cf.~\cref{prop:sympu}) that $u^-(Q)=\infty$. 
\end{remark}

\begin{remark}\label{alggroups}
The inequality $u^-(D)\leqslant \widetilde{u}(D)$ from \cref{prop:start} is a reflection of the fact that in characteristic $2$, any (twisted) orthogonal group is contained in an associated  (twisted)  symplectic group, and an orthogonal group is anisotropic if the associated symplectic group is anisotropic (see \cite[(17.3.6)]{LAG} for the case of an orthogonal group associated to a quadratic form over a field. The general case is similar. See also \cite[(2.15)]{AGNT} for details the on correspondence between the isotropy of forms and the associated groups in characteristic different from $2$. The result in characteristic $2$ is similar). 

 Further, \cref{symreps}  can be thought of as a classification of those orthogonal groups whose associated symplectic group is anisotropic. That is, the symplectic group associated to an orthogonal group is anisotropic if and only if the orthogonal group is isomorphic to the  group of isometries of a nonsingular generalised quadratic form over a division algebra with involution  that does not represent any symmetric elements. 
\end{remark}

%

We now give an  upper bound for $u(D)$ in terms of our hermitian $u$-invariants.

\begin{thm}\label{bounds} Assume $D$ is not a field.
 Then
$$ u(D)\leqslant  u^+(D)+u^+_d(D)\,.$$
\end{thm}
\begin{proof}
 Let $\rho$ be an anisotropic quadratic form over $(D,\theta)$. Let $\rho_1$ and $\rho_2$ be anisotropic nonsingular and totally singular forms respectively such that $\rho\simeq \rho_1\perp \rho_2$. By \cref{diagstuff}, $(c)$, there exists $w_1,\ldots, w_n\in \Sym(D,\theta)\setminus \Alt(D,\theta)$ such that $\rho_2\simeq \qf{w_1,\ldots, w_n}$. 

Consider $\rho_1$. If  $\rho_1$ represents an element  $y\in \Sym(D,\theta)$ then by \cite[(8.2)]{dolphin:singuni} there exists an $x\in D\setminus \Sym(D,\theta)$ and a nonsingular quadratic form $\rho'$ over $(D,\theta)$ such that $\rho_1\simeq \rho'\perp\qf{x,x+y}$. Repeating this argument, we can assume that 
$$\rho_1\simeq \rho'\perp\qf{x_1,x_1+y_1}\perp \ldots \perp \qf{x_m,x_m+y_m}$$
for some quadratic form $\rho'$ over $(D,\theta)$ that does not represent any element in $\Sym(D,\theta)$, $x_1,\ldots, x_m\in D\setminus\Sym(D,\theta)$ and $y_1,\ldots, y_m\in \Sym(D,\theta)$.  By \cref{repsym} we have that  $\dim_D(\rho')\leqslant u^-(D)$.
For all $d\in D^\times$ and $i=1,\ldots, m$ we have that $\qf{x_i,x_i+y_i}$ represents $\theta(d)y_id$. Therefore by the anisotropy of $\rho$ we have that the totally singular form $\qf{y_1,\ldots, y_m, w_1,\ldots, w_n}$ is anisotropic. That is, $m+n\leqslant u^+_d(D)$ by \cref{totsing}. Therefore $$\dim_D(\rho)=\dim_D(\rho') + 2m+n\leqslant u^-(D)+2u^+_d(D)=u^+(D)+u^+_d(D)$$
by \cref{prop:start}, $(1)$.
\end{proof}


We also record the following upper bound on $u^+_d(D)$, which we further investigate in \cref{directandtotdecomp}.

\begin{prop}\label{directbound}  Let $D$ be an $F$-division algebra.
Then $$u_d^+(D)\leqslant  \frac{[F:F^2]}{\deg(D)}\,.$$ 
\end{prop}
\begin{proof}
Let $\varphi$ be a direct hermitian form over $(D,\theta)$ and let $(A,\s)=\Ad(\varphi)$.
By 
\cite[(3.1)]{Knus:1998}, $D$ is of order $2$ in the Brauer group, and hence by 
\cite[(9.1.4)]{Gille:2006} splits over a separable algebraic extension $K/F$.
 By \cref{prop:sympalt}, $(c)$ and   \cref{sepdir} we have that $(A,\s)_K$ is anisotropic.  Therefore by \cref{prop:sympalt}, $(a)$ there exists an anisotropic symmetric bilinear form $\delta$ over $K$ such that $(A,\s)_K\simeq \Ad(\delta)$. As $\dim_F(\delta)=\deg(A)$,   by \cref{splitcase} we have 
$$ \dim_D(\varphi)=  \mathrm{coind}(A) = \frac{\deg(A)}{\deg(D)}=\frac{\dim_F(\delta)}{\deg(D)}\leqslant \frac{[K:K^2]}{\deg(D)} \,. $$
By \cite[(1.3)]{Baeza:uinv2} we have $[F:F^2]=[K:K^2]$, and hence 
$u_d^+(D)\leqslant \frac{[F:F^2]}{\deg(D)}$. 
\end{proof}

\section{Direct forms and totally decomposable algebras}\label{directandtotdecomp}

In this section we show that the inequality  in \cref{directbound} is in fact an equality  for an important class of division algebras. This generalises the fact that $u^+(F)=[F:F^2]$ from \cref{splitcase}. 
First we recall certain facts about quaternion algebras for use in this and following sections. 
 An $F$-\emph{quaternion algebra} is a central simple $F$-algebra of degree $2$. It follows that quaternion algebras are either split or division. 
We call a central simple $F$-algebra $A$ \emph{totally decomposable} if there exists $F$-quaternion algebras $Q_1,\ldots, Q_n$ such that $A\simeq Q_1\otimes_F\ldots\otimes_F Q_n$.

 Let $Q$ be an $F$-quaternion algebra.
 By  \cite[(2.21)]{Knus:1998}, the map $Q\rightarrow Q,$  $x\mapsto \Trd_Q(x)-x$ is the unique symplectic involution $\gamma$ on $Q$; it is called the \emph{canonical involution of $Q$}.
For all $x\in Q$ we  have $\Nrd_Q(x)=\gamma(x)x$.
Direct computation shows that $\Alt(Q,\gamma)=F$.
Any $F$-quaternion algebra has an $F$-basis $(1,i,j,k)$ such that
\begin{eqnarray}\label{eqnarray:qatbas}i^2 -i =a, j^2=b\,\textrm{ and }\, k=ij=j-ji\,,\end{eqnarray}
for some  $a\in F$  and $b\in F^\times$
 (see  \cite[Chap.~IX, Thm.~26]{Albert:1968});  such a basis is called an \emph{$F$-quaternion basis}.
 Conversely, for  $a\in F$  and $b\in F^\times$
 the above relations uniquely determine an $F$-quaternion algebra (up to $F$-isomorphism), which we denote by $[a,b)_F$. 
By the above, up to isomorphism any $F$-quaternion algebra is of this form.

By \cite[(7.1)]{Knus:1998}, for any $F$-algebra with orthogonal involution $(A,\s)$ with $\deg(A)$ even and  any $x,y\in \Alt(A,\s)$ we have $\Nrd_A(x)F^{\times 2}=\Nrd_A(y) F^{\times 2}$.
 Therefore, as in \cite[\S7]{Knus:1998}, we may make the following definition.  The \emph{determinant of $(A,\sigma)$}, denoted $\det(A,\sigma)$, is the square class of the reduced norm of an arbitrary  alternating unit, that is
$\det(A,\sigma)= \mathrm{Nrd}_A(x)\cdot F^{\times2}$  in  $F^\times/ F^{\times2}$  for any $x\in \Alt(A,\sigma)\cap A^\times.$

\begin{prop}[{\cite[(7.5)]{dolphin:orthpfist}}] \label{pfistinv} Let  $(Q_1,\tau_1), \ldots, (Q_n,\tau_n)$ be $F$-algebras with orthogonal involution, $b_\ell=\det(Q_\ell,\tau_\ell)$ for $\ell=1,\ldots, n$ and $(A,\s)=(Q_1,\tau_1)\otimes\ldots\otimes(Q_n,\tau_n)$.
 Then $(A,\s)$ is anisotropic if and only if $\pff{b_1,\ldots, b_n}$ is anisotropic.  
\end{prop}

\begin{thm}\label{totdecomp}
Let $D$ be a totally decomposable $F$-division algebra. Then $$u^+_d(D)=\frac{[F:F^2]}{\deg(D)}\,.$$
\end{thm}
\begin{proof}
Let $Q_1,\ldots,Q_m$ be $F$-quaternion division algebras such that $$D\simeq Q_1\otimes_F\ldots\otimes_F Q_m$$ and let $a_1,\ldots, a_m\in F$ and $b_1,\ldots, b_m\in F^\times$ be such that for $\ell=1,\ldots,m$ we have $Q_\ell\simeq [a_\ell,b_\ell)_F$. Further let  
 $j_\ell\in Q^\times$ be such that $j_\ell^2=b_\ell$ for $\ell=1,\ldots,m$.
 If  $[F:F^2]\leq \infty$ let $n\in \nat$ be such that $[F:F^2]=2^n$. In this case 
we already have $u_d^+(D)\leqslant  \frac{2^n}{\deg(D)} $ by \cref{directbound}. Otherwise take  $n\in \nat$ with $n> m$ and arbitrarily large.

First we show that the bilinear Pfister form $\varphi=\pff{b_1,\ldots, b_m}$ over $F$ is anisotropic, which in particular implies $m\leqslant n$.
For $\ell=1,\ldots, m$, let $\gamma_\ell$ be the canonical involution on $Q_\ell$ and let  $\tau_\ell=\Int(j_\ell)\circ\gamma_\ell$. As $j_\ell\in \Alt(Q_\ell,\tau_\ell)$ by (\ref{alt}) we have $\det(Q_\ell,\tau_\ell)=b_\ell$ for all $\ell=1,\ldots, m$. 
Let $(D,\tau)=\bigotimes_{\ell=1}^m(Q_\ell,\tau_\ell)$.
As 
$D$ is division $(D,\tau)$ is anisotropic.
 Therefore by \cref{pfistinv} we have that $\varphi$ is anisotropic.
%
%
%
If $m=n$ then we are done as $\mathrm{coind}(D)=1$ and hence by \cref{invint}  $$u^+_d(D)\geqslant 1 = \frac{2^n}{\deg(D)}\,.$$
Otherwise, since  $[F:F^2]=2^n$ or $\infty$, we can   find $b_{m+1},\ldots, b_{n}\in F^\times$ such that the bilinear Pfister form $\psi=\pff{b_1,\ldots, b_n}$ over $F$ is anisotropic and $F=F^2(b_1,\ldots, b_n)$. 

For $\ell=m+1,\ldots, n$ let $Q_\ell=M_2(F)$  and $j_\ell\in Q_\ell$ such that $j_\ell^2=b_\ell$. For $\ell=m+1,\ldots, n$, let $\gamma_\ell$ be the canonical involution on $Q_\ell$ and let  $\tau_\ell=\Int(j_\ell)\circ\gamma_\ell$. As 
above we have  $\det(Q_\ell,\tau_\ell)=b_\ell$ for all $\ell=m+1,\ldots, n$. Therefore the $F$-algebra with involution $(A,\s)=\bigotimes_{\ell=1}^{n}(Q_\ell,\tau_\ell)$ is anisotropic by the anisotropy of $\psi$ and \cref{pfistinv}. 

Finally,  as $D$ is division and  $Q_\ell$ is split for all $\ell=m+1,\ldots, n$ we have $A\sim_B D$ and 
$$\textrm{coind}(A) = \frac{2^n}{2^m} = \frac{2^n}{\deg(D)}\,.$$
Therefore $u^+_d(D)\geqslant \frac{2^n}{\deg(D)}$ by \cref{invint}, and hence the result.
\end{proof}

Note that division algebras in characteristic $2$ that are not totally decomposable were constructed in \cite[\S3]{rowen:indecomp}.

\section{$u$-invariants of quaternion algebras}\label{quatalgs}

In this section we collect several results on our invariants for quaternion algebras and determine them exactly when the degree of the field extension $F/F^2$ is $2$. This in particular includes both local and global fields of characteristic $2$.

Recall first that every $F$-quaternion algebra $Q$ has an associated quadratic form over $F$.
  By considering $Q$ as an $F$-vector space, we can view $(Q,\mathrm{Nrd}_Q)$ as a $4$-dimensional quadratic  form on $F$, called the \emph{norm form of $Q$}. 
If $Q\simeq[a,b)_F$ for some $a\in F$ and $b\in F^\times$, we have $(Q,\Nrd_Q)\simeq\pfr{b,a}$.

\begin{prop}[{\cite[(12.5)]{Elman:2008}}]\label{lemma:splitting}
Let $Q_1$ and $Q_2$ be quaternion $F$-algebras and let $\pi_i=(Q_i,\Nrd_{Q_i})$ for $i=1,2$. 
 Then
 $Q_1\simeq Q_2$
if and only if  $\pi_1\simeq \pi_2$.
In particular, $Q_i$ is division if and only if $\pi_i$ is anisotropic.
\end{prop}

There is a well-known relationship between nonsingular quadratic forms and alternating hermitian forms over a quaternion algebra, from which we derive the following result. 

\begin{prop} \label{prop:sympu}
Let $Q$ be an $F$-quaternion division algebra. Then
$$u^-(Q) \leqslant\frac{1}{4} \widetilde{u}(F)\,.$$
\end{prop}
\begin{proof}
Let $\gamma$ be the canonical involution on $Q$ and $\varphi=(V,h)$ be  an anisotropic alternating hermitian form over $(Q,\gamma)$. Then by \cite[(4.1) and (4.2)]{dolphin:totdecompsymp}, if we consider $V$ as a vector space over $F$, for the map $q_h:x\mapsto h(x,x)$ for $x\in V$ the pair  $(V,q_h)$ is  an anisotropic nonsingular  quadratic form over $F$ of dimension $4\cdot\dim_D(\varphi)$.  
Hence $u^-(Q) \leqslant\frac{1}{4} \widetilde{u}(F)$.
\end{proof}

\begin{cor}\label{+lessthatsq}
Let $Q$ be an $F$-quaternion division  algebra. Then $$u^+(Q)\leqslant [F:F^2]\quad \textrm{and} \quad u(Q)\leqslant \frac{3}{2}[F:F^2]\,.$$
\end{cor}
\begin{proof}
By \cite[(1.2, $(d)$)]{Baeza:uinv2}, we have $\widetilde{u}(F)\leqslant 2[F:F^2]$. Therefore by \cref{prop:sympu} we have $u^-(Q)\leqslant \frac{1}{2}[F:F^2]$. By \cref{totdecomp} we have $u_d^{+}(Q)=\frac{1}{2}[F:F^2]$. Hence $u^+(Q)\leqslant [F:F^2]$ by \cref{prop:start}, $(1)$. The inequality involving $u(Q)$ follows by \cref{bounds}.
\end{proof}

By  \cite[(1.2, $(d)$)]{Baeza:uinv2} we have $u(F)\leqslant 3[F:F^2]$. Given this, \cref{+lessthatsq} and \cref{directbound} we ask the following question.

\begin{question}\label{qu}
Do we have 
 $$u(D)\leqslant \frac{3[F:F^2]}{\deg(D)}$$
 for $D$ an arbitrary  $F$-division algebra?
\end{question}

Note by \cref{bounds} and  \cref{directbound},  showing that  $u^-(D)\leqslant \frac{[F:F^2]}{\deg(D)}$ for  an arbitrary  $F$-division algebra $D$ would give a positive answer to \cref{qu}.

We now recall the definition of the Arf invariant of a nonsingular quadratic form over an $F$-division algebra with involution $(D,\theta)$, introduced in \cite{Tits:genquadforms}.
We denote the central simple $F$--algebra of $n\times n$ matrices over $D$  by $M_n(D)$. For a matrix $M\in M_n(D)$, let $M^t$ denote the transpose of $M$ and let $M^\ast$ denote the image of $M$  under the $F$--involution on $M_n(D)$ given by
$$\left((a_{ij})_{1\leqslant i, j\leqslant n}\right)^\ast= (\theta(a_{ij})_{1\leqslant i, j\leqslant n})^t\,.$$

Let $\rho=(V,q)$ be an $n$--dimensional nonsingular quadratic form 
over   $(D,\theta)$. By  \cite[Chapter 1, (5.1.1)]{{Knus:1991}} one can find a matrix  $M\in M_n(D)$ such that $q:V\rightarrow D/\Alt(D,\theta)$ is given by $$(x_1,\ldots, x_n)\mapsto (\theta(x_1),\ldots, \theta(x_n))M(x_1,\ldots, x_n)^t +\Alt(D,\theta)\,.$$ 
 If $(V,h)$ is the polar form of $\rho$ then the map   $h:V\times V\rightarrow D$ is given by $$(x_1,\ldots, x_n)\times (y_1,\ldots, y_n)\mapsto (\theta(x_1),\ldots, \theta(x_n))(M+M^\ast)(y_1,\ldots, y_n)^t\,,$$
(see  \cite[Chapter 1, (5.3)]{{Knus:1991}}).
Let $N=M+M^\ast$ be a matrix associated to the polar form of $\rho$. As $\rho$ is nonsingular we have that $N$ is invertible and that at least one of  $\deg(D)$ or $\dim_D(\rho)$ is even. Let $2m=\deg(D)\cdot \dim_D(\rho)$.  We denote the set $\{a^2+a\mid a\in F\}$ by $\wp(F)$. 
The \emph{Arf invariant of $\rho$} is then defined as the class in ${F}/\wp{F}$ given by $$\Srd_{M_n(D)}(N^{-1}\cdot M) + \frac{m(m-1)}{2} +\wp(F).$$
We denote this class by $\Delta(\rho)$. By \cite[Corollaire 4]{Tits:genquadforms}, $\Delta(\rho)$ depends only on the isometry class of $\rho$ and not on the choice of $M$. 
Note  that for all $\lambda\in F^\times$ we have $\Delta(\rho)=\Delta(\lambda\rho)$. 

For a quaternion algebra $Q$ we have that $\Srd_Q=\Nrd_Q$. 
Further, it is easy to show that for $a_1,\ldots, a_n\in Q\setminus\Sym(Q,\gamma)$ we have 
$$ \Delta(\qf{a_1,\ldots, a_n}_{(Q,\gamma)})=\sum_{i=1}^{n}\Nrd_Q(a_i)+\wp(F) \,. $$

We now fix an $F$-quaternion division algebra $Q$ and    $a\in F$ and $b\in F^\times$ such that  $Q=[a,b)_F$ and let $\gamma$ be the canonical involution on $Q$.

\begin{prop}\label{trivarf}
For any  $1$-dimensional nonsingular quadratic form $\rho$ over $(Q,\gamma)$,  $\Delta(\rho)$ is non-trivial.
\end{prop}
\begin{proof}
By  \cref{diagstuff}, we have that  $\rho\simeq \qf{x}$ for some $x\in Q\setminus \Sym(Q,\gamma)$.  
Since $x\notin \Sym(Q,\gamma)$ we have that $\Trd_Q(x)\neq 0$. 
Suppose $\Delta(\rho)$ is trivial, that is $\Nrd_Q(x)\in \wp(F)$.
First assume that $\Trd_Q(x)=1$.
 Let $d\in F$ be such that $\Nrd_Q(x)=d^2+d$. Then $\Nrd_Q(x+d)=0$. Since $Q$ is division it follows that $x+d=0$, and hence $x=d\in F$. Since $\Alt(Q,\gamma)=F$,  by \cref{diagstuff} this contradictions the nonsingularity of  $\rho$. 
 Now assume that $\Trd_Q(x)\neq 1$. There exists  an element $\lambda\in F^\times$ such that $\Trd_Q(\lambda x)=1$. 
 As $\Delta(\rho)=\Delta(\lambda \rho)$, it follows from the first part of the proof that $\Delta(\rho)$ is non-trivial.
\end{proof}

 In \cite[\S5]{bechermah:uinvar} it was shown that  in characteristic different from $2$ there always exist $3$-dimensional anisotropic skew-hermitian forms over a quaternion algebra with canonical involution  except in a few exceptional cases.  We now show a similar result in characteristic $2$ for quadratic forms over $(Q,\gamma)$. Here an anisotropic $3$-dimensional  form always exists, and we  give an explicit example.

\begin{prop}\label{Tsuk}
Any $3$-dimensional nonsingular  quadratic form over $(Q,\gamma)$ with trivial  Arf invariant is anisotropic. 
\end{prop}
\begin{proof}
 Let $\rho$ be a $3$-dimensional nonsingular quadratic form over $(Q,\gamma)$. If $\rho$ is isotropic then
 there exists an $x\in Q\setminus \Sym(Q,\gamma)$
 such that
  $\rho \simeq \qf{x,y,y}$ for all $y\in Q\setminus \Sym(Q,\gamma)$ by \cite[Chapter 1, (6.5.3)]{Knus:1991}. It follows that $\Delta(\rho)=\Delta(\qf{x})$, which is non-trivial by \cref{trivarf}.
\end{proof}
%


\begin{cor}\label{3}
 $\widetilde{u}(Q)\geqslant 3$. 
\end{cor}	
\begin{proof}
Fix an $F$-quaternion basis $(1,i,j,k)$ of $Q$ satisfying the relations (\ref{eqnarray:qatbas}).
Take $\rho=\qf{i+k, i+\frac{aj}{b}, i+\frac{aj}{b} +k}\,.$ This form is nonsingular by \cref{diagstuff}, $(b)$.
 We have \begin{eqnarray*}\Nrd_Q\left(i+k\right)&=&a +ab \\
\Nrd_Q\left( i+\frac{aj}{b} \right)&=&a + \frac{a^2}{b}
 \\
\Nrd_Q\left(i+\frac{aj}{b} +k \right)&=& a+ \frac{a^2}{b} +a+ ab\,.
\end{eqnarray*}
Summing these terms together gives $0$, hence $\Delta(\rho)$ is trivial.  It follows that $\rho$ is anisotropic by \cref{Tsuk}.
\end{proof}

\begin{cor}\label{u=3}
If $[F:F^2]=2$ then $u^+(Q)=2$ and $u(Q)=\widetilde{u}(Q)=3$.
\end{cor}
\begin{proof}
We have that $u^+({Q})\geqslant 2$ by \cref{orth2} and $u^+(Q)\leqslant [F:F^2]=2$ by \cref{+lessthatsq}. We have that $u({Q})\geqslant 3$ by \cref{3} and $u(Q)\leqslant \frac{3}{2}[F:F^2]=3$ by \cref{+lessthatsq}. Also by \cref{3} we have  $3\leqslant\widetilde{u}(Q)\leqslant u(Q)=3$.
\end{proof}

\begin{remark}
\cref{u=3} shows that the conjectured bound from \cref{qu} is met in the case where $[F:F^2]=2$. This is perhaps surprising, as in the case where $D$ is a field, this bound is often quite poor. Indeed, it is not hard to show that if $[F:F^2]=2$ then $u(F)\leqslant 4$, whereas the bound in \cref{qu} gives $u(F)\leqslant 6$. 
\end{remark}

We call a field $F$ of characteristic $2$ a \emph{local field} if $F\simeq\mathbb{F}_{2^n}((t))$. That is, a Laurent series in one variable over  $\mathbb{F}_{2^n}$, the field with $2^n$ elements for some $n\in \mathbb{N}$. 
 In \cite{tsukchar2} it was shown that $u^+(Q)=2$ in the case where $F$ is a local field. \cref{u=3} gives a different proof of this fact,  as $F=F^2(t)$ and hence $[F:F^2]=2$.
Local fields $F$  also have the property that there is a unique $F$-quaternion division algebra  (this follows easily from, for example,  \cite[Chapt.~ XIV, \S 5, Prop.~12]{serre1979local}). In the next section we further investigate fields with this property.

\section{Kaplansky fields}\label{kapfields}

Following  \cite[\S6]{bechermah:uinvar}, we call a field $F$ a \emph{Kaplansky field} if there is a unique $F$-quaternion division algebra (up to isomorphism). 
In \cite[\S6]{bechermah:uinvar}, the hermitian $u$-invariant was shown to be always $3$ or less for a quaternion algebra over a  Kaplansky field of characteristic different from $2$. In this section, we consider our $u$-invariants for a quaternion algebra $Q$ over a Kaplansky field of characteristic $2$. We show that while  $\widetilde{u}(Q)=3$, $u(Q)$ and $u^+(Q)$ can be arbitrarily large. 

For the rest of this section, assume $F$ is a Kaplansky field and for some $a\in F$ and $b\in F^\times$ let $Q=[a,b)_F$  be the unique $F$-quaternion division algebra and $\gamma$ its canonical involution.  Further, let $\pi=\pfr{b,a}$. Then by \cref{lemma:splitting}, $\pi$ is   the unique anisotropic $2$-fold Pfister form over $F$ (up to isometry).

We first calculate $\widetilde{u}(F)$ and $u^-(Q)$. 

\begin{prop}\label{prop:21}
Let $\rho$ be a  quadratic form   over $F$ of  type $(4,1)$. Then $\rho$ is isotropic. 
\end{prop}
\begin{proof}
 We may scale $\rho$ so that $\rho\simeq c[1,d]\perp e[1,f]\perp \qf{1}$ for some $c,d,e,f\in F^\times$.  Let $\psi_1=\qf{1}\perp c[1,d]$ and $\psi_2=\qf{1}\perp e[1,f]$. If either $\psi_1$ or $\psi_2$ are isotropic then so is $\rho$. 

Assume both $\psi_1$ and $\psi_2$ are anisotropic. Note that $\psi_1$ and $\psi_2$ are $3$-dimensional  subforms of the $2$-fold Pfister  forms  $\pi_1= \pfr{c,d}$  and $\pi_2=\pfr{e,f}$ respectively.
For $i=1,2$, if $\pi_i$ is isotropic, then it is hyperbolic, and in particular has a totally isotropic subspace of dimension $2$. Hence any $3$-dimensional subform of $\pi_i$ would be isotropic. Therefore  $\pi_i$ is anisotropic for $i=1,2$. In  particular, we must have that $\pi_i\simeq\pi$ for $i=1,2$ as $\pi$ is the unique anisotropic $2$-fold Pfister form over $F$.  It follows that $\rho$ is a subform of the hyperbolic $3$-fold Pfister form $\pi\perp \pi=\pi_1\perp\pi_2$. 
As $\pi\perp\pi$ has a totally isotropic subspace of dimension $4$ and $\rho$ is of dimension $5$, it follows that $\rho$ is isotropic.
\end{proof}

\begin{cor}\label{nonsingandsymp}
$\widetilde{u}(F)=4$ and $u^-(Q)=1$.
\end{cor}
\begin{proof}
As $\pi$ is anisotropic, we have $\widetilde{u}(F)\geqslant 4$. Let $\rho$ be a nonsingular quadratic form over $F$ of dimension $6$ or more. Then $\rho$  has  a subform of type $(4,1)$. Therefore $\rho$ is isotropic by \cref{prop:21} and hence $\widetilde{u}(F)=4$. 
We have $u^-(Q)\geqslant 1$ by \cref{always2}.  That $u^-(Q)=1$ follows from $\widetilde{u}(F)=4$ and \cref{prop:sympu}.
\end{proof}

We  now compute $\widetilde{u}(Q)$. Generalised quadratic forms were classified over a large class of division algebras in  \cite{tignol:qfskewfield}, from which the following classification result for Kaplansky fields is easily derived.

\begin{prop}\label{arfclass}
Nonsingular quadratic forms over $(Q,\gamma)$ are classified by their Arf invariants.
\end{prop}
\begin{proof} 
As $F$ is a  Kaplansky field we have that the $2$-torsion part of the Brauer group of $F$ consists only of the class of $Q$ and $0$. In particular, 
 any central simple $F$-algebra of exponent $2$ is Brauer equivalent to $Q$ and there is no Cayley division algebra over $F$ (see the comments before \cite[(5.1)]{tignol:qfskewfield}). Therefore by \cite[(5.2)]{tignol:qfskewfield} we have that quadratic forms over $(Q,\gamma)$ are classified by their Arf invariants and a relative invariant  taking values in the Brauer group of $F$ modulo the class generated by $Q$. However, since the class of $Q$  is the only non-trivial class in the Brauer group of $F$, this invariant is always trivial.
 \end{proof}

\begin{cor}\label{simkap} For all 
 $x\in Q\setminus \Sym(Q,\gamma)$ and  $c\in F^\times$ we have $\qf{x}_{(Q,\gamma)}\simeq\qf{cx}_{(Q,\gamma)}$. 
\end{cor}
\begin{proof} 
Note that the forms in the statement are nonsingular by \cref{diagstuff}, $(b)$.
We have $\Delta(\qf{x})=\Delta(\qf{cx})$ and hence the result follows from \cref{arfclass}.
\end{proof}

The following can be seen as a characteristic $2$ analogue  of \cite[Theorem 1]{Tsuk}.

\begin{prop}\label{tsuk}
Let $\rho$ be a  quadratic form over $(Q,\gamma)$  of type $(4,0)$ or $(3,1)$. Then $\rho$ is isotropic. 
\end{prop}
\begin{proof} Let $(1,i,j,k)$ be an $F$-quaternion basis of $Q$ satisfying (\ref{eqnarray:qatbas}). 
By \cref{diagstuff} there exist	$x_1,x_2,x_3\in Q\setminus \Sym(Q,\gamma)$ and an element $y\in Q$ such that $\rho\simeq\qf{x_1,x_2,x_3,y}$. We may assume that $y\in Q^\times$ and further that $x_1,x_2,x_3$ and $y$ lie in the $3$-dimensional $F$-subspace of $Q$ generated by $i,j$ and $k$ by \cref{diagstuff}, $(a)$.  As this subspace is $3$-dimensional, there exists $c_1,\ldots, c_4 \in  F$ not all zero such that $\sum_{\ell=1}^3c_\ell x_\ell+ c_4y=0$. If $c_4\neq 0$,  then for $\ell=1,2,3$  set $d_\ell=\frac{c_\ell}{c_4}$ if $c_\ell\neq 0$ and $d_\ell=1$ otherwise. 
By \cref{simkap}, we have that 
$$\rho\simeq\qf{x_1,x_2,x_3, y}\simeq \qf{d_1x_1,d_2x_2, d_3x_3, y}\,,$$
and hence  $\rho$ is isotropic. 
If $c_4=0$, then $\qf{x_1,x_2,x_3}$ is isotropic by a similar argument. In either case $\rho$ is isotropic. \end{proof}

\begin{cor}\label{nonsingukap}
$\widetilde{u}(Q)=3$. 
\end{cor}
\begin{proof}
We have that $\widetilde{u}(Q)\geqslant 3$ by \cref{3}.
It follows from \cref{tsuk} that $\widetilde{u}(Q)\leqslant 3$.
\end{proof}



We now consider $u^+(Q)$ and $u(Q)$ for Kaplansky fields. 

\begin{prop}\label{herkap}
$u^+(Q)=\frac{1}{2}[F:F^2]+1$.
\end{prop}
\begin{proof}
This follows  from \cref{totdecomp}, \cref{nonsingandsymp} and \cref{prop:start}, $(1)$.
\end{proof}

\begin{thm}\label{max}
$u(Q)=\mathrm{sup}\{ 3, \frac{1}{2}[F:F^2]+1 \}$.
\end{thm}
\begin{proof} Clearly $u(Q)\geqslant \widetilde{u}(Q)$ and $ \widetilde{u}(D)=3$ by  \cref{nonsingukap}. We also have that $u(Q)\geqslant u^+(Q)$ by \cref{prop:start}, $(3)$.
We now show that $u(Q)\leqslant\mathrm{sup}\{ 3, u^+(Q)\}$, which implies the result by \cref{herkap}.
 By \cref{prop:start}, $(1)$ and \cref{nonsingandsymp}, we also have $u^+(Q)=u^+_d(Q)+1$.

Let $\rho$ be an anisotropic quadratic form over $(Q,\gamma)$ such that $\dim_Q(\rho)=u(Q)$. Let $n,m\in\mathbb{N}$ be such that $\rho$ is of type 
 $(n,m)$. In particular 
  $n+m=u(Q)$. We have that  $n< 4$ by  \cref{tsuk}.
    Assume $n=3$. Then $m=0$ by  \cref{tsuk} and hence $u(Q)=3$. 
    We must also have $u^+(Q)\leqslant 3$ as $u^+(Q)\leqslant u(Q)$. Hence the result in this case.

  Assume $n=2$. Then as $u(Q)\geqslant 3$ we must have $m\geqslant 1$. 
Let $\rho_1$ and $\rho_2$ be respectively  nonsingular and totally singular anisotropic forms over $(Q,\gamma)$ such that $\rho\simeq \rho_1\perp\rho_2$. Then as $u^-(Q)=1$ by 
 \cref{nonsingandsymp}, $\rho_1$ represents an element $x\in\Sym(Q,\gamma)$ by \cref{repsym}.
 Since $\rho_1$ is anisotropic, we have $x\notin \Alt(Q,\gamma)$. 
It follows from the anisotropy of $\rho$ that the totally singular form $\rho_2\perp \qf{x}_{(Q,\gamma)}$ is anisotropic.
Therefore  by \cref{directtotsing} there exists a direct hermitian form over $(Q,\gamma)$ with dimension $\dim_Q(\rho_2)+1$. Therefore  $u^+_d(Q)\geqslant m+1$ and hence 
$$u^+(Q)=u^+_d(Q)+1  \geqslant m+2 = \dim_Q(\rho) = u(Q)\geqslant 3 \,. $$ 

Assume now that $n=1$ or $0$. As we know $u(Q)\geqslant 3$, in this case we must have $m\geqslant 2$ or $3$ respectively. Then there exists  an anisotropic  totally singular quadratic form $\rho'$ over $(Q,\gamma)$ with $\dim_Q(\rho')= u(Q)-1$ or $u(Q)$ respectively. By \cref{directtotsing} we therefore have 
$$u^+(Q)= u^+_d(Q)+1 \geqslant u(Q)\geqslant 3\,,$$
as required.\end{proof}

We  now show that    Kaplansky fields $F$  with  unique $F$-quaternion division algebra $Q$  can be constructed with invariants  $u^+(Q)$, $u(Q)$ and $u(F)$  arbitrarily large.

Let $\rho=(V,q)$ be a nonsingular quadratic form over a field  $K$.
If $\dim_K(\rho)\geqslant 3$ or if $\rho$ is anisotropic and $\dim(\rho)=2$, then we call the function field of the projective quadric over $K$ given by $\rho$ the \emph{function field of $\rho$} and denote it by $K(\rho)$. In the remaining cases we set $K(\rho)=K$.
This agrees with the definition in \cite[\S22]{Elman:2008}.

\begin{lemma}\label{kapcon}
Let $K$ be a field of characteristic $2$ such that there exists a $K$-quaternion division algebra $H$. Then there exists a field extension $L/K$ such that 
\begin{enumerate}[$(i)$]
\item $L$ is a Kaplansky field with unique $L$-quaternion division algebra $H_L$,
\item $[L:L^2]\geqslant [K:K^2]$.
\end{enumerate}
\end{lemma}
\begin{proof} Let  $\rho$ be the norm form of $H$ over $K$. This form is anisotropic  by \cref{lemma:splitting}.
Let  $S_0$ be the set of all
anisotropic $2$-fold Pfister forms over $K$.
 Choose a well-ordering on the set $S_0$ and index its elements by ordinal numbers. So for some ordinal $\alpha$, we have $S_0=\{\psi_i\mid i<\alpha\}$. We construct a field $K^1$ by transfinite induction as follows: let $K_0=K$ and define 
\begin{itemize}
\item $K_i=K_{i-1}(\psi_i)$  if $i$ is not a limit ordinal and $(\psi_i)_{K_{i-1}}\not\simeq \rho_{K_{i-1}}$,
\item  $K_i=K_{i-1}$  if $i$ is not a limit ordinal and $(\psi_i)_{K_{i-1}}\simeq \rho_{K_{i-1}}$,
\item $K_i=\bigcup_{j<i}K_j$  if $i$ is a limit ordinal.
\end{itemize}
We then set $K^{[1]}=K_\alpha$.
For all $i$, if $\varphi$ is  an anisotropic symmetric bilinear  form over $K_{i-1}$ 
we have that $\varphi_{K{i}}$  is also anisotropic by \cite[(10.2.1)]{knebusch:1969}.
Hence it follows by transfinite induction that for all anisotropic symmetric bilinear forms $\varphi$ over $K$ the form $\varphi_{K^1}$ is anisotropic.  
Assume that  $ \rho_{K_{i-1}}$ is anisotropic. If  $(\psi_i)_{K_{i-1}}\simeq \rho_{K_{i-1}}$, clearly  $\rho_{K_{i}}$ is anisotropic. If $(\psi_i)_{K_{i-1}}\not\simeq \rho_{K_{i-1}}$ then 
by \cite[(23.6)]{Elman:2008} 
 $\rho_{K_{i}}$ is isotropic if and only if $\rho_{K_{i-1}}\simeq (\lambda \psi_i)_{{K_{i-1}}}$ for some $\lambda\in K^\times$. It follows from \cite[(9.9)]{Elman:2008} that  $\rho_{K_{i-1}}\simeq (\lambda \psi_i)_{K_{i-1}}$ if and only if $\rho_{K_{i-1}}\simeq (\psi_i)_{K_{i-1}}$.
 This contradicts $(\psi_i)_{K_{i-1}}\not\simeq \rho_{K_{i-1}}$, hence 
   $\rho_{K_{i}}$ is anisotropic. It follows  by transfinite induction that $\rho_{K^1}$ is anisotropic.

Let $S_1$ be the set of  all $2$-fold Pfister  forms over $K^{[1]}$ and construct $(K^{[1]})^{[1]}=K^{[2]}$ by the same procedure. Repeating this process, for $n\geqslant 1$, let $K^{[n]}={(K^{[n-1]})}^{[1]}$ and let $L=\bigcup_{n=1}^\infty K^{[n]}$.  We have that  
for all anisotropic symmetric bilinear forms $\varphi$ over $F$ the form $\varphi_{L}$ is anisotropic by the same argument given above for $K^{[1]}$.  In particular $[L:L^2]\geqslant [K:K^2]$.
Similarly, $\rho_L$ is anisotropic, and further, is the only anisotropic $2$-fold Pfister form over $L$.
Hence  by \cref{lemma:splitting}, $L$ is a Kaplansky field with unique $L$-quaternion division algebra $H_L$.
\end{proof}

One can also use a similar construction method to that in \cref{kapcon} to prove the following result. As the proof uses the function fields of totally singular quadratic forms there are certain subtleties in this arguement that do not occur in \cref{kapcon}. However, another similar  result  to the one below using totally singular forms but giving a field $L$ with $[L:L^2]=2$ is shown  in  \cite[(6.4)]{totlink}, and  the proof is not substantially  different. 
We leave the full details  to the reader. 

\begin{lemma}\label{fix2power}
Let $K$ be a field of characteristic $2$. Then for any  $n\in\mathbb{N}$ with $n\geqslant 2$ there exists a field extension $L/K$ such that $[L:L^2]=2^n$ and  for all anisotropic $2$-fold Pfister forms $\pi$ over $K$, the form $\pi_L$ is anisotropic.\end{lemma}
\begin{sproof}
If $[K:K^2]< 2^n $ then we can adjoin  variables to $K$ until we have a field $K'$ such that $[K':K']\geqslant 2^n$. For any anisotropic $2$-fold Pfister forms $\pi$ over $K$, $\pi_{K'}$ is anisotropic by \cite[(19.6)]{Elman:2008}. 

Assume now that $[K:K^2]\geqslant 2^n$. Then we can construct a field extension $K^{[1]}/K$ using successive function fields of totally singular quadratic forms over $K$ of dimension $2^n+1$ so that all totally singular quadratic forms over $K$ of dimension $2^n+1$ or greater  are isotropic after extending scalars to $K^{[1]}$ (cf.~\cite[(6.4)]{totlink}). Repeating this process and taking the union of the resulting fields  similarly to  as  in \cref{kapcon} we get a field extension $L/K$ such that $[L:L^2]\leqslant 2^n$. Further,  using $2$-power separation theorem \cite[(1.1)]{2-powersep}  and transfinite induction we obtain that $[L:L^2]=2^n$ and  for all anisotropic $2$-fold Pfister forms $\pi$ over $K$, the form $\pi_L$ is anisotropic. 
\end{sproof}

\begin{prop}\label{bigus}
Let $K$ be a field of characteristic $2$
such that there exists a non-trivial quadratic  separable extension of $K$.
 Then for any $n\in\mathbb{N}$ with $n\geqslant 1$ there exists a field extension $L/K$ such that $L$ is a Kaplansky field with $u(L)\geqslant 2^{n+1}$ and for the unique $L$--quaternion division algebra $H$ we have   $u(H)=u^+(H)= 2^{n}+1$.  
\end{prop}
\begin{proof} Let $c\in K$ be such that $K(x)/K$ is a non-trivial  quadratic separable extension where we have  $x^2+x=c$. Take $K'=K(Y)$ for some variable $Y$ and let $H=[c,Y)_{K'}$. This is an  $K'$-division algebra by \cite[(19.6)]{Elman:2008} and   \cref{lemma:splitting}. 

By \cref{fix2power} and \cref{lemma:splitting}, for any $n\geqslant 2$ we can find a field extension $L/K'$ such that $H_L$ is division and $u^+(L)=[L:L^2]=2^{n+1}\leqslant u(L)$. If $L$ is a Kaplansky field then $u^+(H)= 2^{n}+1$ and $u(H)= 2^{n}+1$  by \cref{herkap} and \cref{max}. Otherwise by \cref{kapcon} there exists a field extension $L'/L$ such that $L'$ is a Kaplansky field with $H_{L'}$ as the unique $L'$-quaternion algebra and $[L':(L')^2]\geqslant [L:L^2]$. 
If $[L':(L')^2]= [L:L^2]$ then we are done. Otherwise repeating  the above constructions  inductively gives the required field.
\end{proof}

\section*{Acknowledgements}
 This work was supported by  the projects   \emph{Explicit Methods in Quadratic Form Theory} and 
 \emph{Automorphism groups of locally finite trees} (G011012), both with the Research Foundation, Flanders, Belgium (F.W.O.~Vlaanderen). %

\small{
}

\end{document}